\documentclass[preprint,11pt]{elsarticle}
\usepackage{amsthm,amssymb,amsmath,graphicx}
\usepackage{epsfig}
\usepackage{color}

\usepackage[figurename=Fig.]{caption}

\newtheorem{theorem}{Theorem}[section]
\newtheorem{corollary}[theorem]{Corollary}

\newtheorem{proposition}[theorem]{Proposition}
\newtheorem{prob}[theorem]{Problem}
\newtheorem{conj}[theorem]{Conjecture}

\theoremstyle{ex}
\theoremstyle{definition}
\theoremstyle{remark}

\newtheorem{remark}[theorem]{Remark}

\newcommand{\R}{\mathbb{R}}

\journal{ADAM}

\begin{document}
\begin{frontmatter}
\title{Open problems in the spectral theory of signed graphs}
\author[1]{Francesco Belardo\corref{cor1}}
\ead{fbelardo@unina.it}
\cortext[cor1]{Corresponding author}
\author[2]{Sebastian M. Cioab\u{a}}
\ead{cioaba@udel.edu}
\author[3]{Jack Koolen}
\ead{koolen@ustc.edu.cn}
\author[4]{Jianfeng Wang}
\ead{jfwang@sdut.edu.cn}

\address[1]{\scriptsize Department of Mathematics and Applications, University of Naples Federico II, I-80126 Naples, Italy}
\address[2]{\scriptsize Department of Mathematical Sciences, University of Delaware, USA}
\address[3]{\scriptsize School of Mathematical Sciences, University of Science and Technology of China, Wen-Tsun Wu Key Laboratory of the Chinese Academy of Sciences, Anhui, 230026, China}
\address[4]{\scriptsize School of Mathematics and Statistics, Shandong University of Technology, Zibo, China}

\date{\today}

\begin{abstract}
Signed graphs are graphs whose edges get a sign $+1$ or $-1$ (the signature). Signed graphs can be studied by means of graph matrices extended to signed graphs in a natural way. Recently, the spectra of signed graphs have attracted much attention from graph spectra specialists. One motivation is that the spectral theory of signed graphs elegantly generalizes the spectral theories of unsigned graphs. On the other hand, unsigned graphs do not disappear completely, since their role can be taken by the special case of balanced signed graphs.

Therefore, spectral problems defined and studied for unsigned graphs can be considered in terms of signed graphs, and sometimes such generalization shows nice properties which cannot be appreciated in terms of (unsigned) graphs. Here, we survey some general results on the adjacency spectra of signed graphs, and we consider some spectral problems which are inspired from the spectral theory of (unsigned) graphs.\\
\end{abstract}


\end{frontmatter}

\section{Introduction}

A \emph{signed graph} $\Gamma=(G,\sigma)$ is a graph $G=(V,E)$, with vertex set $V$ and edge set $E$, together with a function $\sigma:E\rightarrow\{+1,-1\}$ assigning a positive or negative sign to each edge. The (unsigned) graph $G$ is said to be the underlying graph of $\Gamma$, while the function $\sigma$ is called the signature of $\Gamma$. Edge signs are usually interpreted as $\pm1$. In this way, the adjacency matrix $A(\Gamma)$ of $\Gamma$ is naturally defined following that of unsigned graphs, that is by putting $+1$ or $-1$ whenever the corresponding edge is either positive or negative, respectively. One could think about signed graphs as weighted graphs with edges of weights in $\{0,1,-1\}$, however the two theories are very different. In fact, in signed graphs the product of signs has a prominent role, while in weighted graphs it is the sum of weights that is relevant. A walk is positive or negative if the product of corresponding weights is positive or negative, respectively. Since cycles are special kinds of walks, this definition applies to them as well and we have the notions of positive and negative cycles.
\smallskip

Many familiar notions related to unsigned graphs directly extend to signed graphs. For example, the degree $d_v$ of a vertex $v$ in $\Gamma$ is simply its degree in $G$. A vertex of degree one is said to be a pendant vertex. The diameter of $\Gamma=(G,\sigma)$ is the diameter of its underlying graph $G$, namely, the maximum distance between any two vertices in $G$. Some other definitions depend on the signature, for example, the positive (resp.,  negative) degree of a vertex is the number of positive (negative) edges incident to the vertex, or the already mentioned sign of a walk or cycle. A signed graph is {\it balanced} if all its cycles are positive, otherwise it is {\it unbalanced}. Unsigned graphs are treated as (balanced) signed graphs where all edges get a positive sign, that is, the {\em all-positive signature}.

An important feature of signed graphs is the concept of {\em switching} the signature. Given a signed graph $\Gamma=(G,\sigma)$ and a subset $U\subseteq V(G)$, let $\Gamma^U$ be the signed graph obtained from $\Gamma$ by reversing the signs of the edges in the cut $[U,V(G)\setminus U]$, namely $\sigma_{\Gamma^U}(e)=-\sigma_{\Gamma}(e)$ for any edge $e$ between $U$ and $V(G)\setminus U$, and $\sigma_{\Gamma^U}(e)=\sigma_{\Gamma}(e)$ otherwise. The signed graph $\Gamma^U$ is said to be (switching) equivalent to $\Gamma$ and $\sigma_{\Gamma^U}$ to $\sigma_\Gamma$, and we write $\Gamma^U \sim \Gamma$ or $\sigma_{\Gamma^U} \sim \sigma_{\Gamma}$. It is not difficult to see that each cycle in $\Gamma$ maintains its sign after a switching. Hence, $\Gamma^U$ and $\Gamma$ have the same positive and negative cycles. Therefore, the signature is determined up to equivalence by the set of positive cycles (see \cite{zas2}). Signatures equivalent to the all-positive one (the edges get just the positive sign) lead to balanced signed graphs: all cycles are positive. By $\sigma\sim+$ we mean that the signature $\sigma$ is equivalent to the all-positive signature, and the corresponding signed graph is equivalent to its underlying graph. Hence, all signed trees on the same underlying graph are switching equivalent to the all-positive signature. In fact, signs are only relevant in cycles, while the edge signs of bridges are irrelevant.

Note that (unsigned) graph invariants are preserved under switching, but also by vertex permutation, so we can consider the isomorphism class of the underlying graph. If we combine switching equivalence and vertex permutation, we have the more general concept of switching isomorphism of signed graphs. For basic results in the theory of signed graphs, the reader is referred to Zaslavsky \cite{zas2} (see also the dynamic survey \cite{dynamic1}). \smallskip

We next consider matrices associated to signed graphs. For a signed graph $\Gamma=(G,\sigma)$ and a graph matrix $M=M(\Gamma)$, the $M$-polynomial is $\phi_M(\Gamma, x)=\det(xI-M(\Gamma))$. The spectrum of $M$ is called the $M$-spectrum of the signed graph $\Gamma$. Usually, $M$ is the adjacency matrix $A(\Gamma)$ or the Laplacian matrix $L(\Gamma)=D(G)-A(\Gamma)$, but in the literature one can find their normalized variants or other matrices. In the remainder, we shall mostly restrict to $M$ being the adjacency matrix $A(\Gamma)$. The adjacency matrix $A(\Gamma)=(a_{ij})$ is the symmetric $\{0,+1,-1\}$-matrix such that $a_{ij}=\sigma(ij)$ whenever the vertices $i$ and $j$ are adjacent, and $a_{ij}=0$ otherwise. As with unsigned graphs, the Laplacian matrix is defined as $L(\Gamma)=D(G)-A(\Gamma)$, where $D(G)$ is the diagonal matrix of vertices degrees (of the underlying graph $G$). In the sequel we will mostly restrict to the adjacency matrix.

Switching has a matrix counterpart. In fact, let $\Gamma$ and $\Gamma^U$ be two switching equivalent graphs. Consider the matrix $S_U=\textrm{diag}(s_1,s_2,\ldots,s_n)$ such that
\[s_i=\left\{
\begin{array}{ll}
+1,& i\in U;\\
-1, & i\in \Gamma\setminus U.
\end{array}
\right.\]
The matrix $S_U$ is the {\it switching matrix}. It is easy to check that
\[A(\Gamma^U)=S_U\, A(\Gamma)\, S_U, \,\,\,\,\,\,{\rm and}\,\,\,\,\,\,\, L(\Gamma^U)=S_U\, L(\Gamma)\, S_U.\]
Hence, signed graphs from the same switching class share similar graph matrices by means of signature matrices (signature similarity). If we also allow permutation of vertices, we have signed permutation matrices, and we can speak of (switching) isomorphic signed graphs. Switching isomorphic signed graphs are cospectral, and their matrices are signed-permutationally similar. From the eigenspace viewpoint, the eigenvector components are also switched in signs and permuted. Evidently, for each eigenvector, there exists a suitable switching such that all components become nonnegative.

In the sequel, let $\lambda_1(\Gamma)\geq\lambda_2(\Gamma)\geq\cdots\geq\lambda_n(\Gamma)$ denote the eigenvalues of the adjacency matrix $A(\Gamma)$ of the signed graph $\Gamma$ of order $n$; they are all real since $A(\Gamma)$ is a real symmetric matrix. The largest eigenvalue $\lambda_1(\Gamma)$ is sometimes called the \emph{index} of $\Gamma$. If $\Gamma$ contains at least one edge, then $\lambda_1(\Gamma)>0>\lambda_n(\Gamma)$ since the sum of the eigenvalues is $0$. Note that in general, the index $\lambda_1(\Gamma)$ does not equal the spectral radius $\rho(\Gamma)=\max\{|\lambda_i|:1\leq i\leq n\}=\max\{\lambda_1,-\lambda_n\}$ because the Perron--Frobenius Theorem is valid only for the all-positive signature (and those equivalent to it). For example, an all $-1$ signing (all-negative signature) of the complete graph on $n\geq 3$ vertices will have eigenvalues $\lambda_1=\dots =\lambda_{n-1}=1$ and $\lambda_n=-(n-1)$.

We would like to end this introduction by mentioning what may be the first paper on signed graph spectra \cite{zas2}. In that paper, Zaslavsky showed that $0$ appears as an $L$-eigenvalue in connected signed graphs if and only if the signature is equivalent to the all-positive one, that is, $\Gamma$ is a balanced signed graph.

For notation not given here and basic results on graph spectra, the reader is referred to \cite{CvDS,cve}, for some basic results on the spectra of signed graphs, to \cite{zas1}, and for some applications of spectra of signed graphs, to \cite{Gallier}.\smallskip

In Section 2, we survey some important results on graph spectra which are valid in terms of the spectra of signed graphs. In Section 3 we collect some open problems and conjectures which are open at the writing of this note.

\section{What do we lose with signed edges?}

From the matrix viewpoint, when we deal with signed graphs we have symmetric $\{0,1,-1\}$-matrices instead of just symmetric $\{0,1\}$-matrices. Clearly, the results coming from the theory of nonnegative matrices can not be applied directly to signed graphs. Perhaps the most important result that no longer holds for adjacency matrices of signed graphs is the Perron--Frobenius theorem. We saw one instance in the introduction and we will see some other consequences of the absence of Perron--Frobenius in the next section. Also, the loss of non-negativity has other consequences related to counting walks and the diameter of the graph (Theorem \ref{signdiam}). On the other hand, all results based on the symmetry of the matrix will be still valid in the context of signed graphs with suitable modifications. In this section, we briefly describe how some well-known results are (possibly) changed when dealing with matrices of signed graphs.\smallskip

We start with the famous {\em Coefficient Theorem}, also known as {\em Sachs Formula}. This formula, perhaps better than others, describes the connection between the eigenvalues and the combinatorial structure of the signed graph. It was given for unsigned graphs in the 1960s independently by several researchers (with different notation), but possibly first stated by Sachs (cf. \cite[Theorem 1.2]{CvDS} and the subsequent remark). The signed-graph variant can be easily given as follows. An elementary figure is the graph $K_2$ or $C_n$ ($n\geq3$). A basic figure (or linear subgraph, or sesquilinear subgraph) is the disjoint union of elementary figures. If $B$ is a basic figure, then denote by $\mathcal{C}(B)$ the class of cycles in $B$, with $c(B)=|\mathcal{C}(B)|$, and by $p(B)$ the number of components of $B$ and define $\sigma(B)=\prod_{C\in \mathcal{C}(B)}\sigma(C)$. Let $\mathcal{B}_i$ be the set of basic figures on $i$ vertices.
\begin{theorem}[Coefficient Theorem]\label{Coefficient Theorem}
Let $\Gamma$ be a signed graph and let $\phi(\Gamma,x)=\sum_{i=0}^n a_i x^{n-1}$ be its adjacency characteristic polynomial. Then
\begin{equation*}
a_i=\sum_{B\in \mathcal{B}_i}(-1)^{p(B)}2^{c(B)}\sigma(B).
\end{equation*}
\end{theorem}

Another important connection between the eigenvalues and the combinatorial structure of a signed graph is given by the forthcoming theorem. If we consider unsigned graphs, it is well known that the $k$-th spectral moment gives the number of closed walks of length $k$ (cf. \cite[Theorem 3.1.1]{cve}). Zaslavsky \cite{zas1} observed that a signed variant holds for signed graphs as well, and from his observation we can give the subsequent result.
\begin{theorem}[Spectral Moments]\label{Spectral_moments}
Let $\Gamma$ be a signed graph with eigenvalues $\lambda_1\geq \dots \geq \lambda_n$. If $W^{\pm}_k$ denotes the difference between the number of positive and negative closed walks of length $k$, then
\[W^{\pm}_k=\sum_{i=1}^n\lambda_i^k.\]
\end{theorem}

Next, we recall another famous result for the spectra of graphs, that is, the {\em Cauchy Interlacing Theorem}. Its general form holds for principal submatrices of real symmetric matrices (see \cite[Theorem 1.3.11]{cve}). It is valid in signed graphs without any modification to the formula. For a signed graph $\Gamma=(G,\sigma)$ and a subset of vertices $U$, then $\Gamma-U$ is the signed graph obtained from $\Gamma$ by deleting the vertices in $U$ and the edges incident to them. For $v\in V(G)$, we also write $\Gamma-v$ instead of $\Gamma-\{v\}$. Similar notation applies when deleting subsets of edges.
\begin{theorem}[Interlacing Theorem for Signed Graphs]\label{interlacing vertex}
Let $\Gamma=(G,\sigma)$ be a signed graph. For any vertex $v$ of $\Gamma$,
\[\lambda_1(\Gamma)\geq\lambda_1(\Gamma-v)\geq\lambda_2(\Gamma)\geq\lambda_2(\Gamma-v)\geq\cdots\geq\lambda_{n-1}(\Gamma-v)\geq\lambda_n(\Gamma).\]
\end{theorem}

In the context of subgraphs, there is another famous result which is valid in the theory of signed graphs. In fact, it is possible to give the characteristic polynomial as a linear combination of vertex- or edge-deleted subgraphs. Such formulas are known as {\em Schwenk's Formulas} (cf. \cite[Theorem 2.3.4]{cve}, see also \cite{belisi}). As above, $\Gamma-v$ ($\Gamma-e$) stands for the signed graph obtained from $\Gamma$ in which the vertex $v$ (resp., edge $e$) is deleted. Also, to make the formulas consistent, we set $ \phi(\emptyset,x)=1$.

\begin{theorem}[Schwenk's Formulas]\label{schwenk}
Let $\Gamma$ be a signed graph and $v$ (resp., $e=uv$) one of its vertices (resp., edges). Then
\begin{eqnarray*}
\phi(\Gamma,x) &=& x \phi(\Gamma-v,x)-\sum_{u \sim v}\phi(\Gamma-u-v,x)-2\sum_{{C}\in \mathcal{C}_v}\sigma{{(C)}}\phi(\Gamma-{C},x), \\
\phi(\Gamma,x) &=&  \phi(\Gamma-e,x)-\phi(\Gamma-u-v,x)-2\sum_{{C}\in \mathcal{C}_e}\sigma{{(C)}}\phi(\Gamma-{C},x),
\end{eqnarray*}
where $\mathcal{C}_a$ denotes the set of cycles passing through $a$.
\end{theorem}

Finally, a natural question is the following: if we fix the underlying graph, how much can the eigenvalues change when changing the signature? Given a graph with cyclomatic number $\xi$, then there are at most $2^{\xi}$ nonequivalent signatures as for each independent cycle one can assign either a positive or a negative sign. However, among the $2^{\xi}$ signatures, some of them might lead to switching isomorphic graphs, as we see later. In general, the eigenvalues coming from each signature cannot exceed in modulus the spectral radius of the underlying graph, as is shown in the last theorem of this section.

\begin{theorem}[Eigenvalue Spread]\label{spread}
For a signed graph $\Gamma=(G,\sigma)$, let $\rho(\Gamma)$ be its spectral radius. Then $\rho(\Gamma)\leq \rho(G)$.
\end{theorem}
\begin{proof}
Clearly, $\rho(\Gamma)$ equals $\lambda_1(\Gamma)$ or $-\lambda_n(\Gamma)$. Let $A$ be the adjacency matrix of $(G,+)$, and $A_\sigma$ be the adjacency matrix of $\Gamma=(G,\sigma)$. For a vector $X=(x_1,\ldots,x_n)^T$, let $|X|=(|x_1|,\ldots,|x_n|)^T$.

If $X$ is a unit eigenvector corresponding to $\lambda_1({A_\sigma})$, by the Rayleigh quotient we get
\[
\lambda_1(G,\sigma)=X^TA_\sigma X\leq |X|^TA|X|\leq \max_{z:z^Tz=1}z^TAz=\lambda_1(G,+).
\]
Similarly, if $X$ is a unit eigenvector corresponding to the least eigenvalue $\lambda_n({A_\sigma})$, by the Rayleigh quotient we get
\[
\lambda_n(G,\sigma)=X^TA_\sigma X\geq |X|^T(-A)|{X}|\geq \min_{z:z^Tz=1}z^T(-A)z=\lambda_n(G,-)=-\lambda_1(G,+).
\]
By gluing together the two inequalities, we get the assertion. \end{proof}

It is evident from the preceding results that the spectral theory of signed graphs well encapsulates and extends the spectral theory of unsigned graphs. Perhaps, we can say that adding signs to the edges just gives more variety to the spectral theory of graphs. This fact was already observed with the Laplacian of signed graphs, which nicely generalizes the results coming from the Laplacian and signless Laplacian theories of unsigned graphs. It is worth mentioning that thanks to the spectral theory it was possible to give matrix-wise definitions of the signed graph products \cite{gehaza}, line graphs \cite{besi,zas1} and subdivision graphs \cite{besi}.

\section{Some open problems and conjectures}
In this section we consider some open problems and conjectures which are inspired from the corresponding results in the spectral theory of unsigned graphs. We begin with the intriguing concept of ``sign-symmetric graph'' which is a natural signed generalization of the concept of bipartite graph.

\subsection{Symmetric spectrum and sign-symmetric graphs}

One of the most celebrated results in the adjacency spectral theory of (unsigned) graphs is the following.
\begin{theorem}\label{least_eigenvalue}
\begin{enumerate}
\item A graph is bipartite if and only if its adjacency spectrum is symmetric with respect to the origin.
\item A connected graph is bipartite if and only if its smallest eigenvalue equals the negative of its spectral radius.
\end{enumerate}
\end{theorem}
For the first part, one does not need Perron--Frobenius theorem. To the best of our knowledge, Perron--Frobenius is crucial for the second part (see \cite[Section 3.4]{BH} or \cite[Section 8.8]{GR} or \cite[Ch. 31]{VLW}).

On the other hand, in the larger context of signed graphs the symmetry of the spectrum is not a privilege of bipartite and balanced graphs. A signed graph $\Gamma=(G,\sigma)$ is said to be {\em sign-symmetric} if $\Gamma$ is switching isomorphic to its negation, that is, $-\Gamma=(G,-\sigma)$. It is not difficult to observe that the signature-reversal changes the sign of odd cycles but leaves unaffected the sign of even cycles. Since bipartite (unsigned) graphs are odd-cycle free, it happens that bipartite graphs are a special case of sign-symmetric signed graphs, or better to say, if a signed graph $\Gamma=(G,\sigma)$ has a bipartite underlying graph $G$, then $\Gamma$ and $-\Gamma$ are switching equivalent. In Fig. \ref{fig1} we depict an example of a sign-symmetric graph. Here and in the remaining pictures as well negative edges are represented by heavy lines and positive edges by thin lines.

\unitlength=1mm
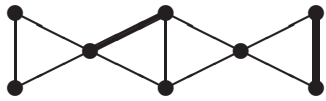
\begin{figure}[h]
  \centering
\begin{picture}(40,17)(,7)
\thicklines
\put(0,10){\circle*{2.0}}\put(0,20){\circle*{2.0}}\put(10,15){\circle*{2.0}}\put(20,10){\circle*{2.0}}
\put(20,20){\circle*{2.0}}\put(30,15){\circle*{2.0}}\put(40,10){\circle*{2.0}}\put(40,20){\circle*{2.0}}
\put(0,10){\line(0,1){10}} \put(10,15){\line(2,-1){10}} \put(10,15){\line(-2,1){10}} \put(10,15){\line(-2,-1){10}} \put(0,10){\line(0,1){10}} \put(30,15){\line(2,1){10}} \put(30,15){\line(2,-1){10}} \put(30,15){\line(-2,-1){10}} \put(30,15){\line(-2,1){10}} \put(20,10){\line(0,1){10}}

\put(40,10){\line(0,1){10}}\put(40.15,10){\line(0,1){10}} \put(39.85,10){\line(0,1){10}} \put(39.65,10){\line(0,1){10}} \put(40.25,10){\line(0,1){10}}
\put(10,15){\line(2,1){10}} \put(10,15.15){\line(2,1){10}} \put(10,14.85){\line(2,1){10}} \put(10,15.3){\line(2,1){10}} \put(10,14.70){\line(2,1){10}}
\end{picture}
\caption{A sign-symmetric signed graph.}
\label{fig1}
\end{figure}

\newpage
If $\Gamma$ is switching isomorphic to $-\Gamma$, then $A$ and $-A$ are similar and we immediately get:
\begin{theorem}
Let $\Gamma$ be a sign-symmetric graph. Then its adjacency spectrum is symmetric with respect to the origin.
\end{theorem}
The converse of the above theorem is not true, and counterexamples arise from the theory of Seidel matrices. The Seidel matrix of a (simple and unsigned) graph $G$ is $S(G)=J-I-2A$, so that adjacent vertices get the value $-1$ and non-adjacent vertices the value $+1$. Hence, the Seidel matrix of an unsigned graph can be interpreted as the adjacency matrix of a signed complete graph. The signature similarity becomes the famous Seidel switching. The graph in Fig. \ref{count} belongs to a triplet of simple graphs on 8 vertices sharing the same symmetric Seidel spectrum but not being pairwise (Seidel-)switching isomorphic. In \cite[p. 253]{ssc}, they are denoted as $A_1$, its complement $\bar{A_1}$ and $A_2$ (note, $A_2$ and its complement $\bar{A_2}$ are Seidel switching isomorphic). In fact, $A_1$ and its complement $\bar{A_1}$ are cospectral but not Seidel switching isomorphic. In terms of signed graphs, the signed graph $A'_1$ whose adjacency matrix is $S(A_1)$ has symmetric spectrum but it is not sign-symmetric.

\unitlength=1mm
\begin{figure}[h]
  \centering
\begin{picture}(40,38)(,-5)
\thicklines
\put(0,0){\circle*{2.0}}\put(10,15){\circle*{2.0}}\put(0,30){\circle*{2.0}}\put(20,0){\circle*{2.0}}\put(20,30){\circle*{2.0}}\put(30,15){\circle*{2.0}}\put(40,0){\circle*{2.0}}\put(40,30){\circle*{2.0}}
\put(0,0){\line(2,3){20}} \put(40,0){\line(-2,3){20}} \put(0,0){\line(1,0){40}} \put(10,15){\line(1,0){20}} \put(10,15){\line(2,-3){10}} \put(30,15){\line(-2,-3){10}}
\end{picture}
\caption{The graph $A_1$.}
\label{count}
\end{figure}
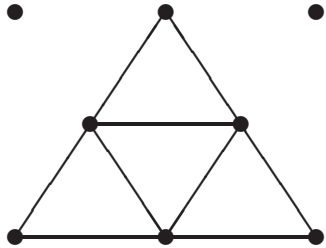
Note that the disjoint union of sign-symmetric graphs is again sign-symmetric. Since the above counterexamples involve Seidel matrices which are the same as signed complete graphs, the following is a natural question.
\begin{prob}
Are there non-complete connected signed graphs whose spectrum is symmetric with respect to the origin but they are not sign-symmetric?
\end{prob}

Observe that signed graphs with symmetric spectrum have odd-indexed coefficients of the characteristic polynomial equal to zero and all spectral moments of odd order are also zero. A simple application of Theorems \ref{Coefficient Theorem} and \ref{Spectral_moments} for $i=3$ or $k=3$, respectively, leads to equal numbers of positive and negative triangles in the graph. When we consider $i=5$ or $k=5$, we cannot say that the numbers of positive and negative pentagons are the same. The following corollary is an obvious consequence of the latter discussion (cf. also \cite[Theorem 1]{ssc}).

\begin{corollary}
A signed graph containing an odd number of triangles cannot be sign-symmetric.
\end{corollary}

\begin{remark}\label{rem1}
As we mentioned in Section 2, a signed graph with cyclomatic number $\xi$ has exactly $2^\xi$ not equivalent signatures (see also \cite{mal}). On the other hand, the symmetries, if any, in the structure of the underlying graph can make several of those signatures lead to isomorphic signed graphs.
\end{remark}

\subsection{Signed graphs with few eigenvalues}

There is a well-known relation between the diameter and the number of distinct eigenvalues of an unsigned graph (cf. \cite[Theorem 3.3.5]{cve}). In fact, the number of distinct eigenvalues cannot be less than the diameter plus 1. With signed graphs, the usual proof based on the minimal polynomial does not hold anymore. Indeed, the result is not true with signed graphs. As we can see later, it is possible to build signed graphs of any diameter having exactly two distinct eigenvalues.

For unsigned graphs, the identification of graphs with a small number of eigenvalues is a well-known problem. The unique connected graph having just two distinct eigenvalues is the complete graph $K_n$. If a graph is connected and regular, then it has three distinct eigenvalues if and only if it is strongly regular (see \cite[Theorem 3.6.4]{cve}). At  the 1995 British  Combinatorial  Conference, Haemers posed  the  problem of finding connected graphs with three eigenvalues which are neither strongly regular nor complete bipartite. Answering Haemers' question, Van Dam \cite{vanDam3} and Muzychuk and Klin \cite{Muz} described some constructions of such graphs. Other constructions were found by De Caen, van Dam and Spence \cite{CaDaSp} who also noticed that the first infinite family nonregular graphs with three eigenvalues already appeared in the work of Bridges and Mena \cite{BrMe}. The literature on this topic contains many interesting results and open problems. For example, the answer to the following intriguing problem posed by De Caen (see \cite[Problem 9]{vanDamDom}) is still unknown.
\begin{prob}
Does a graph with three distinct eigenvalues have at most three distinct degrees?
\end{prob}
Recent progress was made recently by Van Dam, Koolen and Jia \cite{DamKooJia} who constructed connected graphs with four or five distinct eigenvalues and arbitrarily many distinct degrees. These authors posed the following {\em bipartite} version of De Caen's problem above.
\begin{prob}
Are there connected bipartite graphs with four distinct eigenvalues and more than four distinct valencies ?
\end{prob}

For signed graphs there are also some results. In 2007, McKee and Smyth \cite{smyth} considered symmetric integral matrices whose spectral radius does not exceed 2. In their nice paper, they characterized all such matrices and they further gave a combinatorial interpretation in terms of signed graphs. They defined a signed graph to be {\em cyclotomic} if its spectrum is in the interval $[2,-2]$. The maximal cyclotomic signed graphs have exactly two distinct eigenvalues. The graphs appearing in the following theorem are depicted in Fig. \ref{maximal}.

\begin{theorem}\label{cyclotomic}
Every maximal connected cyclotomic signed graph is switching equivalent to one of the following:
\begin{itemize}
  \item For some $k = 3, 4, \ldots $, the $2k$-vertex toroidal tessellation $T_{2k}$;
  \item The $14$-vertex signed graph $S_{14}$;
  \item The $16$-vertex signed hypercube $S_{16}$.
\end{itemize}
Further, every connected cyclotomic signed graph is contained in a maximal one.
\end{theorem}

\begin{figure}[h]
\centering
\includegraphics[width=0.8\textwidth]{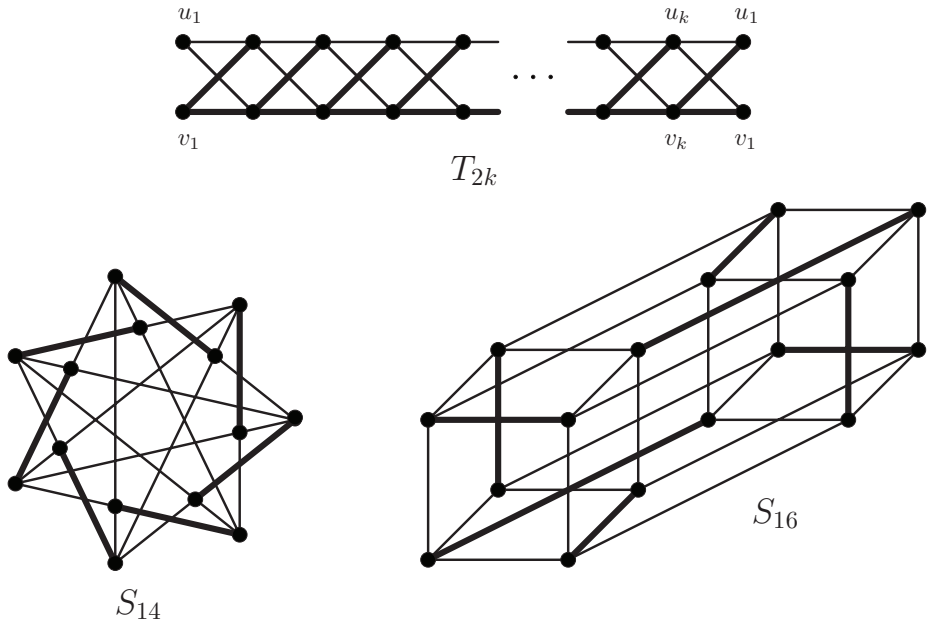}
\caption{Maximal cyclotomic signed graphs.}
\label{maximal}
\end{figure}

It is not difficult to check that all maximal cyclotomic graphs are sign-symmetric. Note that for $k$ even $T_{2k}$ has a bipartite underlying graph, while for $k$ odd $T_{2k}$ has not bipartite underlying graph but it is sign-symmetric, as well. The characteristic polynomial to $T_{2k}$ is $(x-2)^k(x+2)^k$, so $T_{2k}$ is an example of a signed graph with two distinct eigenvalues and diameter $\lfloor\frac{k}{2}\rfloor$.

\begin{prob}[Signed graphs with exactly 2 distinct eigenvalues]
Characterize all connected signed graphs whose spectrum consists of two distinct eigenvalues.
\end{prob}
In the above category we find the complete graphs with homogeneous signatures $(K_n,+)$ and $(K_n,-)$, the maximal cyclotomic signed graphs $T_{2k}$, $S_{14}$ and $S_{16}$, and that list is not complete (for example, the unbalanced $4$-cycle $C_4^-$ and the 3-dimensional cube whose cycles are all negative must be included). There is already some literature on this problem, and we refer the readers to see \cite{gha,Ram}.
All such graphs have in common the property that positive and negative walks of length greater than or equal to 2 between two different and non-adjacent vertices are equal in number. In this way we can consider a {\em signed} variant of the diameter. In a connected signed graph, two vertices are at signed distance $k$ if they are at distance $k$ and the difference between the numbers of positive and negative walks of length $k$ among them is nonzero, otherwise the signed distance is set to $0$. The signed diameter of $\Gamma$, denoted by ${\rm diam}^{\pm}(\Gamma)$, is the largest signed distance in $\Gamma$. Recall that the $(i,j)$-entry of $A^k$ equals the difference between the numbers of positive and negative walks of length $k$ among the vertices indexed by $i$ and $j$. Then we have the following result (cf. \cite[Theorem 3.3.5]{cve}):
\begin{theorem}\label{signdiam}
Let $\Gamma$ be a connected signed graph with $m$ distinct eigenvalues. Then ${\rm diam}^{\pm}(\Gamma)\leq m-1$.
\end{theorem}
\begin{proof}
Assume the contrary, so that $\Gamma$ has vertices, say $s$ and $t$, at signed distance $p\geq m$. The adjacency matrix $A$ of $\Gamma$ has minimal polynomial of degree $m$, and so we may write $A^p = \sum_{k=0}^{m-1} a_k A^k$. This yields the required contradiction because the $(s, t)$-entry on the right is zero, while the $(s, t)$-entry on the left is non-zero.
\end{proof}

Recently, Huang \cite{Huang} constructed a signed adjacency matrix of the $n$-dimensional hypercube whose eigenvalues are $\pm \sqrt{n}$, each with multiplicity $2^{n-1}$. Using eigenvalue interlacing, Huang proceeds to show that the spectral radius (and therefore, the maximum degree) of any induced subgraph on $2^{n-1}+1$ vertices of the $n$-dimensional hypercube, is at least $\sqrt{n}$. This led Huang to a breakthrough proof of the Sensitivity Conjecture from theoretical computer science. We will return to Huang's construction after Theorem \ref{dag_thm}.

\subsection{The largest eigenvalue of signed graphs}

In the adjacency spectral theory of unsigned graphs the spectral radius is the largest eigenvalue and it has a prominent role because of its algebraic features, its connections to combinatorial parameters such as the chromatic number, the independence number or the clique number and for its relevance in applications. There is a large literature on this subject, see \cite{BrNe,CveRo,Hoff1,Hoff2,Nik1,Ste,WCWF} for example.

As already observed, the presence of negative edges leads invalidates of the Perron--Frobenius theorem, and we lose some nice features of the largest eigenvalue:
\begin{itemize}
  \item The largest eigenvalue may not be the spectral radius although by possibly changing the signature to its negative, this can be achieved.
  \item The largest eigenvalue may not be a simple eigenvalue.
  \item Adding edges might reduce the largest eigenvalue.
\end{itemize}
Therefore one might say that it not relevant to study signed graphs in terms of the magnitude of the spectral radius. In this respect, Theorem \ref{interlacing vertex} and Theorem \ref{spread} are helpful because the spectral radius does not decrease under the addition of vertices (together with some incident edges), and the spectral radius of the underlying graph naturally limits the magnitude of the eigenvalues of the corresponding signed graph. For the same reason, the theory of limit points for the spectral radii of graph sequences studied by Hoffman in \cite{Hoff1,Hoff2} is still valid in the context of signed graphs.

The {\em Hoffman program} is the identification of connected graphs whose spectral radii do not exceed some special limit points established by A.J. Hoffman \cite{Hoff2}. The smallest limit point for the spectral radius is $2$ (the limit point of the paths of increasing order), so the first step would be to identify all connected signed graphs whose spectral radius does not exceed $2$. The careful reader notices that the latter question has already been completely solved by Theorem \ref{cyclotomic}. Therefore, the problem jumps to the next significant limit point, which is $\sqrt{2+\sqrt{5}}={\tau}^{\frac{1}{2}}+{\tau}^{-\frac{1}{2}}$, where $\tau$ is the golden mean. This limit point is approached from above (resp., below) by the sequence of positive (resp., negative) cycles with exactly one pendant vertex and increasing girth.

In \cite{BrNe,CvDoGu}, the authors identified all connected unsigned graphs whose spectral radius does not exceed $\sqrt{2+\sqrt{5}}$. Their structure is fairly simple: they mostly consist of paths with one or two additional pendant vertices. Regarding signed graphs, we expect that the family is quite a bit larger than that of unsigned graphs. A taste of this prediction can be seen by comparing the family of Smith Graphs (the unsigned graphs whose spectral radius is $2$, cf. Fig. 2.4 in \cite{CvDS}) with the graphs depicted in Fig. \ref{maximal}. On the other hand, the graphs identified by Cvetkovi\'c et al. acts as a ``skeleton'' (that is, appear as subgraphs) of the signed graphs with the same bound on the spectral radius.

\begin{prob}[Hoffman Program for Signed Graphs]
Characterize all connected signed graphs whose spectral radius does not exceed $\sqrt{2+\sqrt{5}}$.
\end{prob}

\subsection{The smallest eigenvalue of signed graphs}

Unsigned graphs with smallest eigenvalue at least $-2$ have been characterized in a veritable tour de force by several researchers. We mention here Cameron, Goethals, Seidel and Shult \cite{CGSS}, Bussemaker and Neumaier \cite{BN} who among other things, determined a complete list of minimal forbidden subgraphs for the class of graphs with smallest eigenvalue at least $-2$. A monograph devoted to this topic is \cite{cve-2} whose Chapter 1.4 tells the history about the characterization of graphs with smallest eigenvalue at least $-2$.
\begin{theorem}\label{cgss_thm}
If $G$ is a connected graph with smallest eigenvalue at least $-2$, then $G$ is a generalized line graph or has at most $36$ vertices.
\end{theorem}
In the case of unsigned graphs, their work was extended, under some minimum degree condition, from $-2$ to $-1-\sqrt{2}$ by Hoffman \cite{Hoff77} and Woo and Neumaier \cite{WN} and more recently, to $-3$ by Koolen, Yang and Yang \cite{KYY}.

For signed graphs, some of the above results were extended by Vijayakumar \cite{V} who showed that any connected signed graph with smallest eigenvalue less than $-2$ has an induced signed subgraph with at most $10$ vertices and smallest eigenvalue less than $-2$. Chawathe and Vijayakumar \cite{ChVi} determined all minimal forbidden signed graphs for the class of signed graphs whose smallest eigenvalue is at least $-2$. Vijayakumar's result \cite[Theorem 4.2]{V} was further extended by Koolen, Yang and Yang \cite[Theorem 4.2]{KYY} to signed matrices whose diagonal entries can be $0$ or $-1$. These authors introduced the notion of $s$-integrable graphs. For an unsigned graph $G$ with smallest eigenvalue $\lambda_{min}$ and adjacency matrix $A$, the matrix $A-\lfloor \lambda_{min}\rfloor I$ is positive semidefinite. For a natural number $s$, $G$ is called $s$-integrable if there exists an integer matrix $N$ such that $s(A-\lfloor \lambda_{min}\rfloor I)=NN^{T}$. Note that generalized line graphs are exactly the $1$-integrable graphs with smallest eigenvalue at least $-2$.
In a straightforward way, the notion of s-integrabilty can be extended to signed graphs. Now we can extend Theorem \ref{cgss_thm} to the class of signed graphs with essentially the same proof.
\begin{theorem}
Let $\Gamma$ be a connected signed graph with smallest eigenvalue at least $-2$. Then $\Gamma$ is $2$-integrable.
Moreover, if $\Gamma$ has at least $121$ vertices, then $\Gamma$ is $1$-integrable.
\end{theorem}
As $E_8$ has $240$ vectors of (squared) norm $2$, one can take from each pair of such a vector and its negative exactly one to obtain a signed graph on $120$ vertices with smallest eigenvalue $-2$ that is not 1-integrable. Many of these signed graphs are connected.

Koolen, Yang and Yang \cite{KYY}  proved that if a connected unsigned graph has smallest eigenvalue at least $-3$ and valency large enough, then $G$ is $2$-integrable. An interesting direction would be to prove a similar result for signed graphs.
\begin{prob}
Extend \cite[Theorem 1.3]{KYY} to signed graphs.
\end{prob}
An interesting related conjecture was posed by Koolen and Yang \cite{KY}.
\begin{conj}
There exists a constant $c$ such that if $G$ is an unsigned graph with smallest eigenvalue at least $-3$, then $G$ is $c$-integrable.
\end{conj}
Koolen, Yang and Yang \cite{KYY} also introduced $(-3)$-maximal graphs or maximal graphs with smallest eigenvalue $-3$. These are connected graphs with smallest eigenvalue at least $-3$ such any proper connected supergraph has smallest eigenvalue less than $-3$. Koolen and Munemasa \cite{KM} proved that the join between a clique on three vertices and the complement of the McLaughlin graph (see Goethals and Seidel \cite{GoeSei} or Inoue \cite{Inoue} for a description) is $(-3)$-maximal.
\begin{prob}
Construct maximal signed graphs with smallest eigenvalue at least $-3$.
\end{prob}
Woo and Neumaier \cite{WN} introduced the notion of Hoffman graphs, which has proved an essential tool in many results involving the smallest eigenvalue of unsigned graphs (see \cite{KYY}). Perhaps a theory of signed Hoffman graphs is possible as well.
\begin{prob}
Extend the theory of Hoffman graphs to signed graphs.
\end{prob}

\subsection{Signatures minimizing the spectral radius}

As observed in Section 2, an unsigned graph with cyclomatic number $\xi$ gives rise to at most $2^{\xi}$ switching non-isomorphic signed graphs. In view of Theorem \ref{spread}, we know that, up to switching equivalency, the signature leading to the maximal spectral radius is the all-positive one. A natural question is to identify which signature leads to the minimum spectral radius.

\begin{prob}[Signature minimizing the spectral radius]\label{minsr}
Let $\Gamma$ be a simple and connected unsigned graph. Determine the signature(s) $\bar{\sigma}$ such that for any signature $\sigma$ of $\Gamma$, we have $\rho(\Gamma,\bar{\sigma})\leq \rho(\Gamma,\sigma)$.
\end{prob}
This problem has important connections and consequences in the theory of expander graphs. Informally, an expander is a sparse and highly connected graph. Given an integer $d\geq 3$ and $\lambda$ a real number, a $\lambda$-expander is a connected $d$-regular graph whose (unsigned) eigenvalues (except $d$ and possibly $-d$ if the graph is bipartite) have absolute value at most $\lambda$. It is an important problem in mathematics and computer science to construct, for fixed $d\geq 3$, infinite families of $\lambda$-expanders for $\lambda$ {\em small} (see \cite{BL,HLW,MSS15} for example). From work of Alon-Boppana (see \cite{CKNV16,HLW,N}), we know that $\lambda=2\sqrt{d-1}$ is the best bound we can hope for and graphs attaining this bound are called Ramanujan graphs.

Bilu and Linial \cite{BL} proposed the following combinatorial way of constructing infinite families of $d$-regular Ramanujan graphs. A double cover (sometimes called $2$-lift or $2$-cover) of a graph $\Gamma=(G=(V,E),\sigma)$ is the (unsigned) graph $\Gamma'$ with vertex set $V \times \{+1,-1\}$ such that $(x,s)$ is adjacent to $(y,s\sigma(xy))$ for $s = \pm1$.  It is easy to see that if $\Gamma$ is $d$-regular, then $\Gamma'$ is $d$-regular. A crucial fact is that the spectrum of the unsigned adjacency matrix of $\Gamma'$ is the union of the spectrum of the unsigned adjacency matrix $A(G)$ and the spectrum of signed adjacency matrix $A_\sigma=A(\Gamma)$, where $A_\sigma(x,y)=\sigma(x,y)$ for any edge $xy$ of $\Gamma$ and $0$ otherwise (see \cite{BL} for a short proof). Note that this result can be deduced using the method of equitable partitions (see \cite[Section 2.3]{BH}), appears in the mathematical chemistry literature in the work of Fowler \cite{Fowler} and was extended to other matrices and directed graphs by Butler \cite{Butler}.

The spectral radius of a signing $\sigma$ is the spectral radius $\rho(A_\sigma)$ of the signed adjacency matrix $A_\sigma$. Bilu and Linial \cite{BL} proved the important result
\begin{theorem}[Bilu-Linial \cite{BL}]
Every connected $d$-regular graph has a signing with spectral radius at most $c\cdot \sqrt{d\log^3 d}$, where $c>0$ is some absolute constant.
\end{theorem}
\noindent and made the following conjecture.
\begin{conj}[Bilu-Linial \cite{BL}]\label{conj-reg}
Every connected $d$-regular graph $G$ has a signature $\sigma$ with spectral radius at most $2\sqrt{d-1}$.
\end{conj}
If true, this conjecture would provide a way to construct or show the existence of an infinite family of $d$-regular Ramanujan graphs. One would start with a base graph that is $d$-regular Ramanujan (complete graph $K_{d+1}$ or complete bipartite graph $K_{d,d}$ for example) and then repeatedly apply the result of the conjecture above. Recently, Marcus, Spielman and Srivastava \cite{MSS15} made significant progress towards solving the Bilu-Linial conjecture.
\begin{theorem}
Let $G$ be a connected $d$-regular graph. Then there exists a signature $\sigma$ of $G$ such that the largest eigenvalue of $A_\sigma$ is at most $2\sqrt{d-1}$.
\end{theorem}
As mentioned before, $A_{\sigma}$ may have negative entries and one cannot apply the Perron--Frobenius theorem for it. Therefore, the spectral radius of $A_{\sigma}$ is not always the same as the largest eigenvalue of $A_{\sigma}$. In more informal terms, the Bilu-Linial conjecture is about bounding all the eigenvalues of $A_{\sigma}$ by $-2\sqrt{d-1}$ and $2\sqrt{d-1}$ while the Marcus-Spielman-Srivastava result shows the existence of a signing where all the eigenvalues of $A_{\sigma}$ are at most $2\sqrt{d-1}$. By taking the negative of the signing guaranteed by Marcus-Spielman-Srivastava, one gets a signed adjacency matrix where all eigenvalues are at least $-2\sqrt{d-1}$, of course.

There are several interesting ingredients in the Marcus-Spielman-Srivastava result. The first goes back to Godsil and Gutman \cite{GG} who proved the remarkable result that the average of the characteristic polynomials of the all the signed adjacency matrices of a graph $\Gamma$ equals the matching polynomial of $\Gamma$. This is defined as follows. Define $m_0=1$ and for $k\geq 1$, let $m_k$ denote the number of matchings of $\Gamma$ consisting of exactly $k$ edges. The matching polynomial $\mu_{\Gamma}(x)$ of $\Gamma$ is defined as
\begin{equation}
\mu_{\Gamma}(x)=\sum_{k\geq 0}(-1)^km_kx^{n-2k},
\end{equation}
where $n$ is the number of vertices of $\Gamma$. Heilmann and Lieb \cite{HL} proved the following results regarding the matching polynomial of a graph. See Godsil's book \cite{Gac} for a nice, self-contained exposition of these results.
\begin{theorem}\label{Heil-Lieb}
Let $\Gamma$ be a graph.
\begin{enumerate}
\item Every root of the matching polynomial $\mu_{\Gamma}(x)$ is real.
\item If $\Gamma$ is $d$-regular, then every root of $\mu_{\Gamma}(x)$ has absolute value at most $2\sqrt{d-1}$.
\end{enumerate}
\end{theorem}

If $\Gamma$ is a $d$-regular graph, then the average of the characteristic polynomials of its signed adjacency matrices equals its matching polynomial $\mu_{\Gamma}(x)$ whose roots are in the desired interval $[-2\sqrt{d-1},2\sqrt{d-1}]$. As Marcus-Spielman-Srivastava point out, just because the average of certain polynomials has roots in a certain interval, does not imply that one of the polynomials has roots in that interval. However, in this situation, the characteristic polynomials of the signed adjacency matrices form an {\em interlacing family of polynomials} (this is a term coined by Marcus-Spielman-Srivastava in \cite{MSS15}). The theory of such polynomials is developed in \cite{MSS15} and it leads to an existence proof that one of the signed adjacency matrices of $G$ has the largest eigenvalue at most $2\sqrt{d-1}$. As mentioned in \cite{MSS15},
\begin{center}
{\em The difference between our result and the original conjecture is that we do not control the smallest new eigenvalue. This is why we consider bipartite graphs.}
\end{center}
Note that the result of Marcus, Spielman and Srivastava \cite{MSS15} implies the existence of an infinite family of $d$-regular bipartite Ramanujan graphs, but it does not provide a recipe for constructing such family. As an amusing exercise, we challenge the readers to solve Problem \ref{minsr} by finding a signature of the Petersen graph (try it without reading \cite{ZasPet}) or of their favorite graph that minimizes the spectral radius.

A weighing matrix of weight $k$ and order $n$ is a square $n\times n$ matrix $W$ with $0,+1,-1$ entries satisfying $WW^{T}=kI_n$. When $k=n$, this is the same as a Hadamard matrix and when $k=n-1$, this is called a conference matrix. Weighing matrices have been well studied in design and coding theory (see \cite{GerSeb} for example). Examining the trace of the square of the signed adjacency matrix, Gregory \cite{DAG} proved the following.
\begin{theorem}\label{dag_thm}
If $\sigma$ is any signature of $\Gamma$, then
\begin{equation}
\rho(\Gamma,\sigma)\geq \sqrt{k}
\end{equation}
where $k$ is the average degree of $\Gamma$. Equality happens if and only if $\Gamma$ is $k$-regular and $A_{\sigma}$ is a symmetric weighing matrix of weight $k$.
\end{theorem}
This result implies that $\rho(K_n,\sigma)\geq \sqrt{n-1}$ for any signature $\sigma$ with equality if and only if a conference matrix of order $n$ exists. By a similar argument, one gets that $\rho(K_{n,n},\sigma)\geq \sqrt{n}$ with equality if and only if there is a Hadamard matrix of order $n$. Note also that when $k=4$, the graphs attaining equality in the previous result are known from McKee and Smyth's work \cite{smyth} (see Theorem \ref{cyclotomic} above). Using McKee and Smyth characterization and the argument below, we can show that the only $3$-regular graph attaining equality in Theorem \ref{dag_thm} is the $3$-dimensional cube.

Let $Q_n$ denote the $n$-dimensional hypercube. Huang \cite{Huang} constructed a signed adjacency matrix $A_n$ of $Q_n$ recursively as follows:
\begin{equation*}
A_1=\begin{bmatrix} 0&1\\1&0\end{bmatrix}
\text{ and } A_{n+1}=\begin{bmatrix}A_n&I_{2^{n}}\\
I_{2^n}&-A_n
\end{bmatrix},
\end{equation*}
for $n\geq 1$. It is not too hard to show that $A_n=nI_{2^n}$ for any $n\geq 1$ and thus, $A_n$ attains equality in Theorem \ref{dag_thm}. We remark that Huang's method can be also used to produce infinite families of regular graphs and signed adjacency matrices attaining equality in Theorem \ref{dag_thm}. If $G$ is a $k$-regular graph of order $N$ with signed adjacency matrix $A_s$ such that $\rho(A_s)=\sqrt{k}$, then define the $k+1$-regular graph $H$ by taking two disjoint copies of $G$ and adding a perfect matching between them and a signed adjacency matrix for $H$ as $B=\begin{bmatrix} A_s&I_N\\I_N&-A_s\end{bmatrix}$. Because $A_s^2=kI_N$, we can get that $B^{2}=(k+1)I_{2N}$. Thus, using any $4$-regular graph $G$ from McKee and Smyth \cite{smyth} (see again Theorem \ref{cyclotomic}) with a signed adjacency matrix $A_s$ satisfying $A_s^2=4I$, one can construct a $5$-regular graph $H$ with signed adjacency matrix $B$ such that $B^2=5I$. The following is a natural question.
\begin{prob}
Are there any other $5$-regular graphs attaining equality in Theorem \ref{dag_thm} ?
\end{prob}

If the regularity assumption on $G$ is dropped, Gregory considered a the following variant of Conjecture \ref{conj-reg}.
\begin{conj}[\cite{DAG}]\label{conj-nonreg}
If $\Delta$ is the largest vertex degree of a nontrivial graph $G$, then there exists a signature $\sigma$ such that $\rho(G,\sigma)<2\sqrt{\Delta-1}$.
\end{conj}
Gregory came to the above conjecture by observing that in view of Theorem \ref{Heil-Lieb} the bound in the above conjecture holds for the matching polynomial of $G$ and by noticing that
\[\mu_G(x)=\frac{1}{|C|}\sum_{C \in\mathcal{C}}\phi(G,\sigma;x),\]
where $\mathcal{C}$ is the set of subgraphs of $G$ consisting of cycles and $|C|$ is the number of cycles of $C$. Since the matching polynomial of $G$ is the average of polynomials of signed graphs on $G$, one could expect that there is at least one signature $\bar{\sigma}$ such that $\rho(G,\bar{\sigma})$ does not exceed the spectral radius of $\mu_G(x)$. As observed in \cite{DAG}, for odd unicyclic signed graphs the spectral radius of the matching polynomial is always less than the spectral radius of the corresponding adjacency polynomial, but the conjecture still remains valid. We ask the following question whose affirmative answer would imply Conjecture \cite{DAG}.
\begin{prob}
If $\rho$ is the spectral radius of a connected graph $G$, then is there a signature $\sigma$ such that $\rho(G,\sigma)<2\sqrt{\rho-1}$ ?
\end{prob}

In view of the above facts, we expect that the signature minimizing the spectral radius is the one balancing the contributions of cycles so that the resulting polynomial is as close as possible to the matching polynomial. For example, we can have signatures whose corresponding polynomial equals the matching polynomial, as in the following proposition.

\begin{proposition}\label{cactus}
Let $\Gamma$ be a signed graph consisting of $2k$ odd cycles of pairwise equal length and opposite signs. Then $\rho(\Gamma)<2\sqrt{4k-1}.$
\end{proposition}

Is the signature in Proposition \ref{cactus} the one minimizing the spectral radius? We leave this as an open problem (see also \cite{zascycle}).

We conclude this section by observing that for a general graph, it is not known whether Problem \ref{minsr} is NP-hard or not. However, progress is made in \cite{CCCKK} where the latter mentioned problem is shown to be NP-hard when restricted to arbitrary symmetric matrices. Furthermore, the problems described in this subsection can be considered in terms of the largest eigenvalue $\lambda_1$, instead of the spectral radius.

\subsection{Spectral determination problems for signed graphs}

A graph is said to be determined by its (adjacency) spectrum if cospectral graphs are isomorphic graphs. It is well-known that in general the spectrum does not determine the graph, and this problem has pushed a lot of research in spectral graph theory, also with respect to other graph matrices. In general, we can say that there are three kinds of research lines: 1) Identify, if any, cospectral non-isomorphic graphs for a given class of graphs; 2) Routines to build cospectral non isomorphic graphs (e.g., Godsil-McKay switching); 3) Find conditions such that the corresponding graphs are determined by their spectrum.

Evidently, the same problems can be considered for signed graphs with respect to switching isomorphism. On the other hand, when considering signed graphs, there are many more possibilities for getting pairs of switching non-isomorphic cospectral signed graphs. For example, the paths and the cycles are examples of graphs determined by their spectrum, but the same graphs as signed ones are no longer determined by their spectrum since they admit cospectral but non-isomorphic mates \cite{cycle,path}.

Hence, the spectrum of the adjacency matrix of signed graphs has less control on the graph invariants. In view of the spectral moments we get the following proposition:

\begin{proposition}
From the eigenvalues of a signed graph $\Gamma$ we obtain the following invariants:
\begin{itemize}
  \item number of vertices and edges;
  \item the difference between the number of positive and negative triangles ($\frac{1}{6}\sum{\lambda_i^3}$);
  \item the difference between the number of positive and negative closed walks of length $p$ ($\sum{\lambda_i^p}$)
\end{itemize}
\end{proposition}

Contrarily to unsigned graphs, from the spectrum we cannot decide any more whether the graph has some kind of signed regularity, or it is sign-symmetric. For the former, we note that the co-regular signed graph $(C_6,+)$ (it is a regular graph with net regular signature) is cospectral with $P_2\cup \tilde{Q}_4$ (cf. Fig. \ref{C6vsP2Q4}). For the latter, we observe that the signed graphs $A_1$ and $A_2$ are cospectral but $A_1$ is not sign-symmetric while $A_2$ is sign-symmetric.

\unitlength=1mm
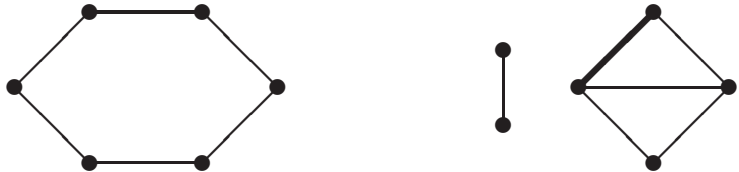
\begin{figure}[h]
  \centering
\begin{picture}(95,35)(,-5)
\thicklines
\put(0,15){\circle*{2.0}}\put(10,5){\circle*{2.0}}\put(10,25){\circle*{2.0}}\put(25,5){\circle*{2.0}}\put(25,25){\circle*{2.0}}\put(35,15){\circle*{2.0}}

\put(65,10){\circle*{2.0}}\put(65,20){\circle*{2.0}} \put(75,15){\circle*{2.0}}\put(85,5){\circle*{2.0}} \put(85,25){\circle*{2.0}}\put(95,15){\circle*{2.0}}
\put(0,15){\line(1,1){10}} \put(0,15){\line(1,-1){10}} \put(10,5){\line(1,0){15}} \put(10,25){\line(1,0){15}} \put(35,15){\line(-1,-1){10}} \put(35,15){\line(-1,1){10}}
\put(65,10){\line(0,1){10}} \put(75,15){\line(1,-1){10}} \put(75,15){\line(1,1){10}} \put(75,15.2){\line(1,1){10}} \put(75,14.8){\line(1,1){10}}\put(75,15.4){\line(1,1){10}} \put(75,14.6){\line(1,1){10}} \put(75,15){\line(1,0){20}} \put(95,15){\line(-1,-1){10}} \put(95,15){\line(-1,1){10}}
\end{picture}
\caption{The cospectral pair $(C_6,+)$ and $P_2\cup \tilde{Q}_4$.}
\label{C6vsP2Q4}
\end{figure}

\subsection{Operations on signed graphs}

In graph theory we can find several operations and operators acting on graphs. For example, we have the complement of a graph, the line graph, the subdivision graph and several kind of products as the cartesian product, and so on. Most of them have been ported to the level of signed graph, in a way that the resulting underlying graph is the same obtained from the theory of unsigned graphs, while the signatures are given in order to preserve the balance property, signed regularities, and in many cases also the corresponding spectra. However, there are a few operations and operators which do not yet have a, satisfactory, 'signed' variant.

One operator that is missing in the signed graph theory is the complement of a signed graph. The complement of signed graph should be a signed graph whose underlying graph is the usual complement, however the signature has not been defined in a satisfactory way yet. What we can ask from the signature of the complement of a signed graph? One could expect some nice features on the spectrum, as for the Laplacian, so that the spectra of the two signed graphs $\Gamma$ and $\bar{\Gamma}$ are complementary to the spectrum of the obtained complete graph.

\begin{prob}
Given a graph $\Gamma=(G,\sigma)$, define the complement $\bar{\Gamma}=(\bar{G},\bar{\sigma})$ such that there are nice (spectral) properties derived from the complete signed graph $\Gamma\cup\bar{\Gamma}$.
\end{prob}

In terms of operators, in the literature we have nice definitions for subdivision and line graphs of signed graphs \cite{besi,zas1}, but for example a satisfactory definition of signed total graph of a signed graph has not yet been provided (but it is considered in the forthcoming paper \cite{bestaza}).

From the product viewpoint, most standard signed graph products have been defined and considered in \cite{gehaza} and the more general NEPS (or, Cvetkovi\'c product) of signed graphs have been there considered. In \cite{germa66} the lexicographic product was also considered, but the given definition is not stable under the equivalence switching classes.

However, there are some graph products which do not have a signed variant yet. As an example, we mention here the wreath product and the co-normal product.

\subsection{Seidel matrices}

The Seidel matrix of a graph $\Gamma$ on $n$ vertices is the adjacency matrix of a signed complete graph $K_n$ in which the edges of $\Gamma$ are negative $(-1)$ and the edges not in $\Gamma$ are positive $(+1)$. More formally, the Seidel matrix $S(\Gamma)$ equals $J_n-I_n-2A(\Gamma)$. Zaslavsky \cite{zas1} confesses that
\begin{center}
{\em This fact inspired my work on adjacency matrices of signed graphs.}
\end{center}
Seidel matrices were introduced by Van Lint and Seidel \cite{VLS} and studied by many people due to their interesting properties and connections to equiangular lines, two-graphs, strongly regular graphs, mutually unbiased bases and so on (see \cite[Section 10.6]{BH} and  \cite{BDKS,GKMS,Sei1} for example). The connection between Seidel matrices and equiangular lines is perhaps best summarized in \cite[p.161]{BH}:
\begin{center}
{\em To find large sets of equiangular lines, one has to find large graphs where the smallest Seidel eigenvalue has large multiplicity.}
\end{center}
Let $d$ be a natural number and $\R^d$ denote the Euclidean $d$-dimensional space with the usual inner product $\langle,\rangle$. A set of $n\geq 1$ lines (represented by unit vectors) $v_1,\dots,v_n\in \R^d$ is called equiangular if there is a constant $\alpha> 0$ such that $\langle v_i,v_j\rangle=\pm \alpha$ for any $1\leq i<j\leq n$. For given $\alpha$, let $N_{\alpha}(d)$ be the maximum $n$ with this property.
The Gram matrix $G$ of the vectors $v_1,\dots,v_n$ is the $n\times n$ matrix whose $(i,j)$-th entry equals $\langle v_i,v_j\rangle$. The matrix $S:=(G-I)/\alpha$ is a symmetric matrix with $0$ diagonal and $\pm 1$ entries off-diagonal. It is therefore the Seidel matrix of some graph $\Gamma$ and contains all the relevant parameters of the equiangular line system. The multiplicity of the smallest eigenvalue $-1/\alpha$ of $S$ is the smallest dimension $d$ where the line system can be embedded into $\R^{d}$.

Lemmens and Seidel \cite{LS} (see also \cite{BDKS,GKMS,JP,LY,Neumaier} for more details) showed that $N_{1/3}(d)=2d-2$ for $d$ sufficiently large and made the following conjecture.
\begin{conj}
If $23\leq d\leq 185$, $N_{1/5}(d)=276$. If $d\geq 185$, then $N_{1/5}(d)=\lfloor 3(d-1)/2\rfloor$.
\end{conj}
The fact that $N_{1/5}(d)=\lfloor 3(d-1)/2\rfloor$ for $d$ sufficiently large was proved by Neumaier \cite{Neumaier} and Greaves, Koolen, Munemasa and Sz\"{o}ll\H{o}si \cite{GKMS}. Recently, Lin and Yu \cite{LY} made progress in this conjecture by proving some
claims from Lemmens and Seidel \cite{LS}. Note that these results can be reformulated in terms of Seidel matrices with smallest eigenvalue $-5$. Seidel and Tsaranov \cite{SeiTsa} classified the Seidel matrices with smallest eigenvalue $-3$.

Neumann (cf. \cite[Theorem 3.4]{LS}) proved that if $N_{\alpha}(d)\geq 2d$, then $1/\alpha$ is an odd integer. Bukh \cite{Bukh} proved that $N_{\alpha}(d)\leq c_{\alpha}d$, where $c_{\alpha}$ is a constant depending only on $\alpha$. Balla, Dr\"{a}xler, Keevash and Sudakov \cite{BDKS} improved this bound and showed that for $d$ sufficiently large and $\alpha\neq 1/3$, $N_{\alpha}(d)\leq 1.93d$. Jiang and Polyanskii \cite{JP} further improved these results and showed that if $\alpha \notin \{1/3,1/5,1/(1+2\sqrt{2})\}$, then $N_{\alpha}(d)\leq 1.49d$ for $d$ sufficiently large. When $1/\alpha$ is an odd integer, Glazyrin and Yu \cite{GY} obtained a general bound $N_{\alpha}(d)\leq \left(2\alpha^2/3+4/7\right)d+2$ for all $n$.

Bukh \cite{Bukh} and also, Balla, Dr\"{a}xler, Keevash and Sudakov \cite{BDKS} conjecture the following.
\begin{conj}
If $r\geq 2$ is an integer, then $N_{\frac{1}{2r-1}}(d)=\frac{r(n-1)}{r-1}+O(1)$ for $n$ sufficiently large.
\end{conj}
When $1/\alpha$ is not a totally real algebraic integer, then $N_{\alpha}(d)=d$. Jiang and Polyanskii \cite{JP} studied the set
$T=\{\alpha \mid \alpha \in (0,1), \limsup_{d \rightarrow \infty}N_{\alpha}(d)/d >1\}$ and showed that the closure of $T$ contains the closed interval $[0, 1/\sqrt{\sqrt{5} +2}]$ using results of Shearer \cite{Shearer} on the spectral radius of unsigned graphs.

Seidel matrices with two distinct eigenvalues are equivalent to regular two-graphs and correspond to equality in the {\em relative bound} (see \cite[Section 10.3]{BH} or \cite{GKMS} for example). It is natural to study the combinatorial and spectral properties of Seidel matrices with three distinct eigenvalues, especially since for various large systems of equiangular lines, the respective Seidel matrices have this property. Recent work in this direction has been done by Greaves, Koolen, Munemasa and Sz\"{o}ll\H{o}si \cite{GKMS} who determined several properties of such Seidel matrices and raised the following interesting problem.
\begin{prob}
Find a combinatorial interpretation of Seidel matrices with three distinct eigenvalues.
\end{prob}
A classification for the class of Seidel matrices with exactly three distinct eigenvalues of order less than $23$ was obtained by Sz\"{o}ll\H{o}si and \"{O}sterg\aa rd \cite{SO}. Several parameter sets for which existence is not known were also compiled in \cite{GKMS}. Greaves \cite{Gre18} studied Seidel matrices with three distinct eigenvalues, observed that there is only one Seidel matrix of order at most $12$ having three distinct eigenvalues, but its switching class does not contain any regular graphs. In \cite{Gre18}, he also showed that if the Seidel matrix $S$ of a graph $\Gamma$ has three distinct eigenvalues of which at least one is simple, then the switching class of
$\Gamma$ contains a strongly regular graph. The following question was posed in \cite{Gre18}.
\begin{prob}
Do there exist any Seidel matrices of order at least $14$ with precisely three distinct eigenvalues whose switching class does not contain a regular graph ?
\end{prob}
The switching class of conference graph and isolated vertex has two distinct eigenvalues. If these two eigenvalues are not rational, then the
switching class does not contain a regular graph. So we suspect that there must be infinitely many graphs whose Seidel matrix has exactly three distinct eigenvalues and its switching graph does not contain a regular graph. A related problem also appears in \cite{Gre18}.
\begin{prob}
Does every Seidel matrix with precisely three distinct rational eigenvalues contain a regular graph in its switching class ?
\end{prob}

The Seidel energy $\mathcal{S}(\Gamma)$ of a graph $\Gamma$ is the sum of absolute values of the eigenvalues of the Seidel matrix $S$ of $\Gamma$. This parameter was introduced by Haemers \cite{Hae2} who proved that $\mathcal{S}(\Gamma)\leq n\sqrt{n-1}$ for any graph $\Gamma$ of order $n$ with equality if and only $S$ is a conference matrix. Haemers \cite{Hae2} also conjectured that the complete graphs on $n$ vertices (and the graphs switching equivalent to them) minimize the Seidel energy.
\begin{conj}
If $\Gamma$ is a graph on $n$ vertices, then $\mathcal{S}(\Gamma)\geq \mathcal{S}(K_n)=2(n-1)$.
\end{conj}
Ghorbani \cite{Gho} proved the Haemers' conjecture in the case $\det(S)\geq n-1$ and very recently, Akbari, Einollahzadeh, Karkhaneei and Nematollah \cite{Akbari1} finished the proof of the conjecture. Ghorbani \cite[p.194]{Gho} also conjectured that the fraction of graphs on $n$ vertices with $|\det S|<n-1$ goes to $0$ as $n$ tends to infinity. This conjecture was also recently proved by Rizzolo \cite{Rizzolo}.

It is known that if $\Gamma$ has even order, then its Seidel matrix $S$ is full-rank. If a graph $\Gamma$ has odd order $n$, then $rank(S)\geq n-1$. There are examples such $C_5$ for example where $rank(S)=n-1$. Haemers \cite{Hae} posed the following problem which is still open to our knowledge.
\begin{prob}
If $rank(S)=n-1$, then there exists an eigenvector of $S$ corresponding to $0$ that has only $\pm 1$ entries ?
\end{prob}
Recently, Van Dam and Koolen \cite{DamKoo19} determined an infinitely family of graphs on $n$ vertices whose Seidel matrix has rank $n-1$ and their switching class does not contain a regular graph.

\section{Conclusions}

Spectral graph theory is a research field which has been very much investigated in the last 30--40 years. Our impression is that the study of the spectra of signed graphs is very far from the level of knowledge obtained with unsigned graphs. So the scope of the present note is to promote investigations on the spectra of signed graphs. Of course, there are many more problems which can be borrowed from the underlying spectral theory of (unsigned) graphs. Here we just give a few of them, but we have barely scratched the surface of the iceberg.

\section*{Acknowledgements} The authors are indebted to Thomas Zaslavsky for reading and revising with many comments an earlier draft of this note. The research of the first author was supported by the University of Naples Federico II under the project ``SGTACSMC''. The research of the second author was supported by the grants NSF DMS-1600768 and CIF-1815922. The research of the third author is partially supported by the National Natural Science Foundation of China (Grant No. 11471009 and Grant No. 11671376) and by Anhui Initiative in Quantum Information Technologies (Grant No. AHY150200). The research of the fourth author was supported by the National Natural Science Foundation of China (No. 11461054).


\end{document}